\documentclass[12pt,reqno]{amsart}
\usepackage{amsthm, amsfonts, amssymb, color}
\usepackage{mathrsfs}
\usepackage{enumerate}
\usepackage{amsmath}
\usepackage{amstext, amsxtra}
\usepackage{txfonts}
\usepackage[colorlinks, linkcolor=black, citecolor=blue, pagebackref, hypertexnames=false, breaklinks=true]{hyperref}
\usepackage{graphicx}
\usepackage[symbol]{footmisc}
\usepackage{latexsym}
\usepackage{tikz}
\usepackage{cite}
\allowdisplaybreaks
\usepackage[breaklinks=true]{hyperref}

\newtheorem{theorem}{Theorem}[section]
\newtheorem{proposition}[theorem]{Proposition}
\newtheorem{corollary}[theorem]{Corollary}
\newtheorem{lemma}[theorem]{Lemma}
\newtheorem{definition}[theorem]{Definition}

\newtheorem{remark}[theorem]{Remark}

\renewcommand\Re{\operatorname{Re}}
\renewcommand\Im{\operatorname{Im}}

\def\Re {\mathfrak{Re\,}}
\def\Im {\mathfrak{Im\,}}

\def \and {{\qquad\text{and}\qquad}}

\numberwithin{equation}{section}
\theoremstyle{definition}

\title[]
{Fractal Remez inequality on the sphere and observability of the heat equation\\
}

\author{Xinyi Chen,\, Shanlin Huang}

\address{Xinyi Chen, School of Mathematics (Zhuhai), Sun Yat-sen University, Zhuhai 519082, Guangdong, China}
\email{chenxy2288@mail2.sysu.edu.cn}

\address {Shanlin Huang (corresponding author), School of Mathematics (Zhuhai), Sun Yat-sen University, Zhuhai 519082, Guangdong, China}
\email{huangshlin6@mail.sysu.edu.cn}


\begin{document}
	
	\begin{abstract}
		This paper is concerned with Remez-type inequalities and their applications in observability inequality. 
		Our aim is twofold. First, we establish the following fractal Remez's inequality on the unit sphere $\mathbb{S}^{n-1}$
		\begin{align*}
			\sup_{\mathbb{S}^{n-1}} |p|\le C(M,N,n,\delta)\sup_{M} |p|,
		\end{align*}
		where $M \subset \mathbb{S}^{n-1}$ (\(n \ge 2\)) is a fractal set of positive \((n-2+\delta)\)-Hausdorff content for arbitrary \(\delta \in (0,1)\), and $p$ is a spherical polynomial of degree at most $N\in \mathbb{Z}^+$.
		Second, building upon this fractal framework, we establish sharp observability inequalities for the heat equation on the sphere, again valid for all \(\delta\in (0, 1)\), which improve the result of Burq and Moyano [J. Eur. Math. Soc. (JEMS), 25 (4) (2023)] in the spherical setting. Furthermore, as an additional application, we prove a lower-dimensional observability inequality for the heat equation with super-quadratic potentials \(V(x) = |x|^{2m}\) ($m \in \mathbb{Z}^+, m\ge 2$) on the whole space \(\mathbb{R}^n\). 
	\end{abstract}
	
	\maketitle
	
	
	\tableofcontents

	
	
	\section{Introduction and main results}\label{sec1}
	
	\subsection{Background and motivation}
	Remez-type inequalities bound a polynomial's maximum on a set by its values on a smaller subset. The classical result  \cite{Remez} states that if \( p \) is a polynomial of degree at most \( N \), and \( \Omega \subset [-1, 1] \) is a measurable set with measure $|\Omega|$, then 
	\begin{align}\label{eq-Remez}
		\sup_{[-1,1]} |p| \leq T_N \left( \frac{4 - |\Omega|}{|\Omega|} \right) \sup_{\Omega} |p|,
	\end{align}
	where \( T_N \) is the Chebyshev polynomial of degree \( N\) defined by
	\begin{align*}
		T_N(x) := \cos(N \arccos x),\quad x \in [-1, 1].
	\end{align*}
	This inequality is a foundational tool in approximation theory (see, e.g., \cite{BE1,BE2}) and has been widely generalized.  For instance,  \cite{Friedland2017sp}  extends it to sets $\Omega$ with positive $\delta$-Hausdorff content, where $\delta>0$, and further extensions are discussed in  \cite{HWW2025logtype}.
	Additionally, multidimensional versions  for convex domains in $\mathbb{R}^n$ have been obtained in \cite{BG,Ganzburg,Kroo}.

	Recently,  Dicke and Veseli\'c \cite{DV2024sphere} extended \eqref{eq-Remez} to the unit sphere $\mathbb{S}^{n-1}$ for spherical polynomials -- that is, the restrictions of polynomials to the sphere (see also the earlier work \cite{Ganzburg}). Specifically, let $p$ be a spherical polynomial of degree at most $N\in \mathbb{Z}^+$ and let $M\subset \mathbb{S}^{n-1}$ be any measurable subset satisfying
	$|M \cap K|>0$, where
	$K=K(x,a)=\{y \in \mathbb{S}^{n-1} : d_s(x, y)=  \arccos(x \cdot y) \leq \pi a,\ x \in \mathbb{S}^{n-1}\}$
	is a spherical cap with radius $a>0$, then
	\begin{align}\label{eq-Remez-measure}
		\sup_{K} |p| \leq \left( c \cdot \frac{|K|}{|K \cap M|} \right)^{2N} \sup_{M \cap K} |p|.
	\end{align}
	Furthermore, Dicke and Veseli\'c applied inequality \eqref{eq-Remez-measure}  to  establish  the following observability inequality for the heat equation on the sphere:  there exists a constant $C_{\text{obs}}>0$ such that, for any initial data $u_0 \in L^2(\mathbb{S}^{n-1})$ and  any $T > 0$, 
	\begin{align}\label{eq-obs-sphe}
		\| e^{T\Delta_{\mathbb{S}^{n-1}}} u_0 \|_{L^2(\mathbb{S}^{n-1})}\leq C_{\text{obs}} \left( \int_0^T \int_E | (e^{t\Delta_{\mathbb{S}^{n-1}}} u_0)(x) |^2 \, dx \, dt \right)^{1/2},
	\end{align}
	holds provided that  $E\subset \mathbb{S}^{n-1}$  is  \(\gamma\)-thick, i.e.,  there exists some radius \(a > 0\) such that \(|E \cap K| \geq \gamma |K|\) for all spherical caps \(K\) of radius \(a\). In particular, any measurable set $E\subset \mathbb{S}^{n-1}$ with positive Lebesgue measure is $\gamma$-thick with $\gamma=|E|/|\mathbb{S}^{n-1}|$.
	
	As a form of quantitative unique continuation, the observability inequality is instrumental in control theory and inverse problems and has been extensively studied in the literature .
	Owing to the vast body of literature on this subject, we only
	refer to \cite{Egidi,Car1,Car2,Lebeau,phung,Wang2019observable,Wang26} that are most relevant to our study. 
	In a recent contrasting development,  Burq and Moyano \cite{Burq} established a new observability inequality on compact manifolds, where the observation region $E$ is allowed to have zero measure.
	More precisely, it was shown that for a $W^{2,\infty}$ compact manifold $(\Omega, g)$ of dimension $d\ge 1$, if \( E \) is a subset of positive \( (d - 1 + \delta) \)-Hausdorff content for some \( \delta \in (0, 1) \) close to 1, 
	then for any $u_0\in L^2(\Omega)$ and any $T>0$, there is a positive constant \( C_{\text{obs}}=C(T, d, \Omega, \delta)>0\) such that
	\begin{align}\label{eq-obs-fractal}
		\|e^{T\Delta_g}u_0\|_{L^2(\Omega)} \leq C_{\text{obs}} \int_0^T \sup_{x \in E} |(e^{t\Delta_g}u_0)(x)| \, dt.
	\end{align}
	However, the result does not cover the scenario where $\delta$ is near 0 for some technical reasons (see Remark \ref{remark-W} (ii)). More recently, \cite{2025heatequation,HWW2025logtype} studied bounded domains \(\Omega\subset \mathbb{R}^n\) with $\delta$ arbitrarily chosen in \( (0, 1)\).

	\textbf{Aim and Motivation.}  Motivated by the above results, the aim of this paper is twofold:
	
	$(i)$ We investigate the validity of \eqref{eq-Remez-measure} when $M\subset\mathbb{S}^{n-1}$ ($n\ge 2$) is a fractal set of positive $(n-2+\delta)$-Hausdorff content, where \(\delta \in (0, 1)\) can be  chosen arbitrarily. 
	
	$(ii)$ By applying the fractal version of  \eqref{eq-Remez-measure},  we establish sharp observability inequalities of the form \eqref{eq-obs-fractal} on the sphere, again with \(\delta \in (0, 1)\)  arbitrarily.  
	Moreover, we prove a lower dimensional observability inequality for the heat equation with confining  potentials $V(x)=|x|^{2m}$  ($m \in \mathbb{Z}^+,m\ge 2$) on the whole space $\mathbb{R}^{n}$. 
	
	\subsection{Main results}\label{sec-main-result}
	
	
	Our first result is stated as follows.
	
	\begin{theorem}\label{THM-smallness on the sphere}
		Let $\delta\in (0,1)$ and let $M\subset \mathbb{S}^{n-1}$  ($n\ge 2$)  be a closed subset satisfying  $C_\mathcal{H}^{n-2+\delta}(M)>0$. 
		Then there exist absolute constants $C_1>0$ and  $C_2=C_2(M,n,\delta)>0$ such that, for every  spherical polynomial $p$ of degree at most $N\in \mathbb{Z}^+$, the following inequality holds
		\begin{align}\label{eq-remez-sphere-whole}
			\sup_{\mathbb{S}^{n-1}} |p|\le C_1 \left ( \left ( \frac{C_2  }{C_\mathcal{H}^{n-2+\delta}(M)}\right )^{\frac{2}{\delta}} N^{\frac{4}{\delta}-1} \right )^{2N}\sup_{M} |p|.
		\end{align}
		
	\end{theorem}

	As applications of Theorem \ref{THM-smallness on the sphere}, we derive the observability inequalities for the parabolic equation
	\begin{align}\label{eq-heat}
		\begin{cases} 
			\partial_t u(t, x) + H u(t, x)=0, & t > 0, \; x \in \Omega, \\ 
			u(0, \cdot) \in L^2(\Omega),
		\end{cases}
	\end{align}
	in the following two settings:
	\begin{itemize}
		\item  {\textit{Case 1:}} $\Omega=\mathbb{S}^{n-1}$ and $H=-\Delta_{\mathbb{S}^{n-1}}$.
		
		\item {\textit{Case 2:}} $\Omega=\mathbb{R}^{n}$ and $H=-\Delta + |x|^{2m}$, where $m \in \mathbb{Z}^+,m\ge 2$.
	\end{itemize}

	More precisely, in the compact setting (\textbf{Case 1}), we establish the following result.
	
	\begin{theorem}\label{THM-observability inequality2}
		In {\textit{Case 1}}, fix any $\delta\in (0,1)$ and let $M\subset \mathbb{S}^{n-1}$ be a closed subset satisfying $C_\mathcal{H}^{n-2+\delta}(M)>0$. Then there exists  a constant 
		$C_{obs}=C(T, n, \delta,M)>0$ such that for every \( u_0 \in L^2(\mathbb{S}^{n-1})\) and \( T > 0 \),
		\begin{align}\label{eq-obs-sphere}
			\left\| e^{T\Delta_{\mathbb{S}^{n-1}}} u_0 \right\|_{L^2(\mathbb{S}^{n-1})} \leq C_{obs} \int_0^T \sup_{x\in M} |(e^{t\Delta_{\mathbb{S}^{n-1}}} u_0)(x)| dt.
		\end{align}
		
	\end{theorem}
	
	Next, we turn to the non-compact setting {\textbf{Case 2}}, we need the following geometric setup. Let 
	\begin{equation}\label{eq-UN-har-526-0}
		\Gamma_1 := \{ x \in \mathbb{R}^n: r_0\le|x|\le r_1, x/|x| \in \Theta_1 \},\ \ 0<r_0<r_1,
	\end{equation}
	where 
	\begin{equation}\label{eq-UN-har-526-1}
		\Theta_1\subset\mathbb{S}^{n-1}\,\,\text{is a closed subset and}\,\,C_\mathcal{H}^{n-2+\delta}(\Theta_1)>0.
	\end{equation}
	\begin{theorem}\label{THM-observability inequality3}
		In {\textit{Case 2}},  fix any $\delta\in (0,1)$ in \eqref{eq-UN-har-526-1}. Then there exists  a constant  $C_{obs}=C(T, n, m, \delta, \Theta_1)>0$ such that for every \( u_0 \in L^2(\mathbb{R}^{n}) \) and  \( T > 0 \),
		\begin{align}\label{eq-obs-ineq-2m}
			\left\| e^{-TH} u_0 \right\|_{L^2(\mathbb{R}^n)} \leq C_{obs} \int_0^T \left ( \int_{r_0}^{r_1}\left ( \sup_{\omega\in \Theta_1}|(e^{-tH} u_0)(r,\omega)| \right )^2 dr  \right )^{\frac{1}{2}}dt.
		\end{align}
	\end{theorem}
	
	The following remarks related to Theorems \ref{THM-smallness on the sphere}--\ref{THM-observability inequality3} are given, respectively.

	\begin{remark}
		(i) Theorem \ref{THM-smallness on the sphere} can be viewed as a fractal analogue of \eqref{eq-Remez-measure}. We mention that \eqref{eq-Remez-measure} yields the   following Logvinenko-Sereda-Kovrijkine type inequality on $\mathbb{S}^{n-1}$: for every \(\gamma\)-thick set \(E \subset \mathbb{S}^{n-1}\) with \(\gamma \in (0, 1]\), every \(\lambda \geq 1\), and every \(f \in \text{Ran } P_{-\Delta_{\mathbb{S}^{n-1}}}((-\infty, \lambda])\), we have
		\begin{equation}\label{eq-LS-Sphe}
			\|f\|_{L^2(\mathbb{S}^{n-1})} \leq \left(\frac{c_1}{\gamma}\right)^{R\lambda^{1/2}+1/2} \|f\|_{L^2(E)}.    
		\end{equation}
		A closely related problem was previously investigated in \cite{Ortega1,Ortega2}. However, the primary objective in those works was to establish uniform estimates with constants independent of  the spectral parameter $\lambda$,  a goal achieved at the cost of imposing stronger assumptions on the sets $E$. 
		In contrast, our approach here does not pursue such uniformity; instead, we derive explicit quantitative bounds that track the precise dependence on  $\lambda$. It plays a crucial role in the application of the well-known Lebeau--Robbiano strategy \cite{Lebeau}, where such precision is essential for the iteration scheme.
		
		(ii) In the spherical setting, Theorem \ref{THM-observability inequality2} improves the result of Burq and Moyano \cite[Theorem 3]{Burq}. Moreover, the condition $C_\mathcal{H}^{n-2+\delta}(M)>0$ is optimal. To see this, consider
		\begin{align*}
			E := \{ x \in \mathbb{S}^{n-1} : \phi_\lambda(x) = 0 \},
		\end{align*}
		where $\phi_\lambda$ is an eigenfunction of $\Delta_{\mathbb{S}^{n-1}}$ on $\mathbb{S}^{n-1}$ associated with the eigenvalue $\lambda > 0$. It is well known that $C_\mathcal{H}^{n-2}(E)>0$,  yet inequality \eqref{eq-obs-sphere} fails for this choice of $E$. 
		Whether Theorem \ref{THM-observability inequality2} extends to arbitrary compact $W^{2,\infty}$ manifolds $(\Omega, g)$ remains an open question.
		
		(iii) To the best of the authors' knowledge, Theorem \ref{THM-observability inequality3} provides the first observability inequality for heat equations with potentials that allows the observation region to be lower-dimensional. In the setting of {\textit{Case 2}},
		it was established in \cite{cone,Miller2008cones} that the conical set
		\begin{equation}\label{eq-cone-0}
			\Gamma_0 := \{ x \in \mathbb{R}^n: |x|\ge r_0, x/|x| \in \Theta_0 \}
		\end{equation}
		is an observable set for \eqref{eq-heat}, where \( r_0 > 0 \) and 
		$\Theta_0\subset \mathbb{S}^{n-1}$ is a nonempty and open. 
		Recently, Martin \cite{Martin} proved that any measurable set with positive measure is an observable set for \eqref{eq-heat}. It is worth noting that the situation is quite different for the case $m=1$, corresponding to Hermite operators, we refer to \cite{AS,JamingSpectralestimates2021,DSV,MP} for further details.
	\end{remark}

	\begin{remark} \label{remark-W}
		(i) Both the proofs of Theorem \ref{THM-observability inequality2} and \ref{THM-observability inequality3} rely crucially on Theorem  \ref{THM-smallness on the sphere}. To prove Theorem  \ref{THM-smallness on the sphere}, we employ the following key technical ingredients: (a) we apply a fractal version of the Tur\'{a}n's lemma established in \cite{{YH}}; (b) we use a slicing theorem related to Riesz capacity, which is  inspired by \cite{Malinnikova1,HWW2025logtype}.
		
		(ii) Our proof of Theorem \ref{THM-observability inequality2} differs from the approach in \cite{Burq}, which is based on the following propagation of smallness estimate for gradients of solutions to \(\operatorname{div}(A(x)\nabla u) = 0\) obtained in \cite{Log-Ma}
		\begin{equation}\label{eq-gre-715}
			\sup_{B_1} |\nabla u| \leq \left(\sup_E |\nabla u|\right)^{\beta} \left(\sup_{B_2} |\nabla u|\right)^{1-\beta},
		\end{equation}
		where \(\beta \in (0, 1)\), and \(E\subset B_2\) is a set with positive \((n-2+\delta)\)-Hausdorff content for some \(\delta \in (0, 1)\) close to \(1\). Here \(B_1\subset B_2\subset\mathbb{R}^n\) are balls of given radii.
		However, whether  \eqref{eq-gre-715} remains valid for \(\delta\) arbitrarily close to  $0$ is currently unknown.

		
		(iii)  Theorem \ref{THM-observability inequality3} constitutes a further refinement of  \cite[Theorem 2.1]{Martin}, achieved with the aid of Theorem \ref{THM-smallness on the sphere}. The key observation is that for the operator \( H = -\Delta + V\) with a radial potential $V$ satisfying $\lim_{r\to+\infty}V(r)= +\infty$,  its eigenfunctions can be expressed as a linear combination of spherical harmonics. This allows us to apply Theorem \ref{THM-smallness on the sphere}, see Subsection \ref{sec-2m} for details.
	\end{remark}

	The remainder of this paper is organized as follows. Section \ref{sec2} is devoted to the proof of Theorem~\ref{THM-smallness on the sphere}. In Section \ref{sec3}, we provide the proofs of Theorems~\ref{THM-observability inequality2}--\ref{THM-observability inequality3}.  Finally, in Appendix~A, we supply the proof of  the asymptotic formula \eqref{eq-eigen-2m} in Lemma \ref{lemma-spectral}.

	\section{Proof of Theorem \ref{THM-smallness on the sphere}}\label{sec2}
	
	\subsection{Preliminaries}\label{sec-pre}\,
	
	This section collects the key concepts and tools that will be used in the proof of Theorem~\ref{THM-smallness on the sphere}. We first recall the definitions of Hausdorff content and Hausdorff measure, which are standard in geometric measure theory (see, e.g., in \cite{Ma,Falconer}). Then we introduce the Riesz $s$-energy and the associated $s$-capacity. Finally, we present slicing theorems for capacity,  which enable us to reduce higher-dimensional problems to lower-dimensional ones.
	
	\begin{definition}[Hausdorff content and measure]
		For $\alpha>0$, the $\alpha$-Hausdorff content of a set $E\subset \mathbb{R}^n$ is defined as 
		\begin{equation}
			C^\alpha_\mathcal{H}(E) := \inf \left\{ \sum_{j=1}^{\infty} d(B_j)^\alpha : E \subset \bigcup_{j=1}^{\infty} B_j \right\}, \quad E \subset \mathbb{R}^n,
		\end{equation}
		where $B_j\subset \mathbb{R}^n$ ($j = 1, 2, \ldots$) are open balls with the diameter \( d(B_j):= diam(B_j) \). The $\alpha$-Hausdorff measure of a set $E\subset \mathbb{R}^n$ is defined as
		\begin{equation}
			\mathcal{H}^\alpha(E) := \lim_{\delta \to 0^+} \inf \left\{ \sum_{j=1}^{\infty} d(B_j)^\alpha : E \subset \bigcup_{j=1}^{\infty} B_j, \, d(B_j) \leq \delta \right\}, \quad E \subset \mathbb{R}^n.
		\end{equation}
	\end{definition}
	It is well-known that the Hausdorff dimension of a set \( A \subset X \) is characterized by the following identities
	\begin{align*}
		\dim_{\mathrm{H}} A :=& \sup\{s : \mathcal{H}^s(A) > 0\} = \sup\{s : \mathcal{H}^s(A) = \infty\}\\
		=& \inf\{t : \mathcal{H}^t(A) < \infty\} = \inf\{t : \mathcal{H}^t(A) = 0\}.
	\end{align*}
	Moreover, it satisfies the monotonicity property
	\begin{equation}
		\dim A \leq \dim B \quad \text{for } A \subset B \subset X,
	\end{equation}
	and the countable stability property
	\begin{equation}
		\dim \bigcup_{i=1}^{\infty} A_i = \sup_i \dim A_i \quad \text{for } A_i \subset X, \, i = 1, 2, \ldots.
	\end{equation}
	
	A basic object in fractal geometry is the $s$-energy of a Borel measure $\mu$: for $s>0$,
	\begin{equation*}
		I_s(\mu) = \iint |x - y|^{-s} \, d\mu x \, d\mu y = \int k_s * \mu \, d\mu,
	\end{equation*}
	where $k_s(x)=|x|^{-s}$ is the Riesz kernel. 
	The following definition of $s$-capacity, formulated in terms of the $s$-energy, will play a crucial role in the subsequent discussion.
	
	\begin{definition}[Riesz $s$-capacity]
		Let \(s > 0\). The (Riesz) \( s \)-capacity of a set \( A \subset \mathbb{R}^n \) is defined by
		\begin{align}
			C_s(A) = \sup \left\{ I_s(\mu)^{-1} : \mu \in \mathcal{M}(A) \text{ with } \mu(\mathbb{R}^n) = 1 \right\},
		\end{align}
		with the interpretation \( C_s(\emptyset) = 0 \), where
		\begin{align}
			\mathcal{M}(A) = \left\{\mu : \mu \text{ is a Radon measure with compact support,} \right. \nonumber\\
			\left. \text{supp } \mu \subset A \text{ and } 0 < \mu(\mathbb{R}^n) < \infty\right\}.
		\end{align}
	\end{definition}
	
	
	
	
	In order to state the slicing theorems, we first recall the notion of the Grassmannian manifold and introduce a natural probability measure on it.
	
	\begin{definition}[Grassmannian manifold and Radon probability measure]
		Let $1 \le m < n$.
		\begin{enumerate}
			\item[(i)] The Grassmannian $G(n,m)$ consists of all $m$-dimensional linear subspaces of $\mathbb{R}^n$. For $V \in G(n,m)$, we denote by $V^\perp \in G(n,n-m)$ its orthogonal complement. For $a \in V^\perp$, set $V_a := V + a$ (the $m$-plane parallel to $V$ through $a$).
			\item[(ii)] The Radon probability measure $\gamma_{n,m}$ on $G(n,m)$ is defined by
			\[
			\gamma_{n,m}(A) := \theta_n\bigl(\{ g \in O(n) : gV \in A \}\bigr) \quad (A \subset G(n,m)),
			\]
			where $O(n)$ is the orthogonal group consisting of all orthogonal transformations on $ \mathbb{R}^n$.  $\theta_n$ is the Haar measure on $O(n)$ such that $\theta_n(O(n))=1$.
		\end{enumerate}
		
		
	\end{definition}
	
	The following slicing lemmas allow us to estimate the capacity of a set in $\mathbb{R}^n$ by integrating the capacities of its intersections with lower-dimensional planes.
	\begin{lemma}(\cite[Chapter 10]{Ma})\label{lemma-slice}
		There is a constant $c$, depending only on $n$ and $m$, such that for $m<s<n$ and for any compact subset $E\subset \mathbb{R}^n$,
		\begin{align}
			c(n,m)C_s(E)\le \int \int_{V^\perp} C_{s-m}(E \cap V_a) \, d\mathcal{H}^{m} a \, d\gamma_{n,n-m} V,
		\end{align}
		where $V\in G(n,n-m)$.
	\end{lemma}
	
	\begin{lemma}(\cite[Chapter 10]{Ma})\label{lemma-Ma-10.8}
		For \( m < s < n \) and \( A \subset \mathbb{R}^n \), if \( C_s(A) > 0 \), then for \( \gamma_{n,n-m} \) almost all \( W \in G(n,n-m) \),
		\begin{align}
			\mathcal{H}^m \bigl( \{ a \in W^\perp : C_{s-m}(A \cap W_a) > 0 \} \bigr) > 0.
		\end{align}
	\end{lemma}
	
	\subsection{Auxiliary set $W$ and slicing arguments}\label{subsec2.2-W}\,
	
	In this subsection, we construct an auxiliary subset $W \subset \mathbb{S}^{n-1}$, which will be used later in the application of the slicing lemmas. The motivation for introducing $W$ is explained in Remark \ref{remark-construct-W}. By applying the slicing lemma to $W$ in Lemma \ref{lemma-back}, we obtain a distinguished spherical segment $\tilde{I}_a$ satisfying several specific properties. This segment will play a key role in the proof of Theorem \ref{THM-smallness on the sphere}. 
	
	\begin{lemma}\label{lemma-construct-W}
		Let $\delta\in (0,1)$. For any spherical polynomial $p$ of degree at most $N\in \mathbb{Z}^+$ and any closed subset $M\subset \mathbb{S}^{n-1}$ with $C_\mathcal{H}^{n-2+\delta}(M)>0$, there exist a set $\widetilde{M}$ satisfying $C_\mathcal{H}^{n-2+\delta}(\widetilde{M})=0$, a point $m \in M \setminus \widetilde{M}$ and a constant $\tilde{r} > 0$ such that, upon defining
		\begin{equation}\label{eq-W-01}
			W:=M\cap K(m,\tilde{r}),  
		\end{equation}
		the following two estimates hold:
		\begin{align}\label{eq-p-W}
			|p(x)|\ge \frac{c_0}{2}\left \| p \right \|_{L^\infty(\mathbb{S}^{n-1})}, \ \ \forall x\in W,
		\end{align}
		and
		\begin{align}\label{W-content}
			C_\mathcal{H}^{n-2+\delta}(W) \ge \tilde{r}^{n-2+\delta}.
		\end{align}
	\end{lemma}
	
	\begin{proof}
		We first establish the existence of a set $\widetilde{M}$ with $C_\mathcal{H}^{n-2+\delta}(\widetilde{M})=0$ and a point $m\in M\setminus \widetilde{M}$ together with a constant $c_0>0$ (independent of $p$) such that
		\begin{align}\label{eq-m}
			|p(m)|\ge c_0 \left \| p \right \|_{L^\infty(\mathbb{S}^{n-1})},
		\end{align}
		and
		\begin{align}\label{eq-subset}
			\limsup_{r \downarrow 0} r^{-n+2-\delta}C_\mathcal{H}^{n-2+\delta}(M \cap K(m,r)) \ge 1.
		\end{align}
		To this end, we recall the definition of the upper \( s \)-densities of \( M \) at a point \( a \):
		\[
		\Theta^{s*}(M, a) = \limsup_{r \downarrow 0} (2r)^{-s} C_\mathcal{H}^s(M \cap K(a, r)),
		\]
		It is known (see, e.g., \cite[pp. 89-91]{Ma}) that
		\begin{align*}
			\Theta^{s*}(M, x)\ge 2^{-s},
		\end{align*}
		for \( C_\mathcal{H}^s \)-almost all \( x \in M \). Consequently, there exists a set  $\widetilde{M}$ with $C_\mathcal{H}^{n-2+\delta}(\widetilde{M})=0$ such that, for all $a\in M\setminus \widetilde{M}$,
		$$
		\limsup_{r \downarrow 0} r^{-n+2-\delta} C_\mathcal{H}^{n-2+\delta}(M \cap K(a,r)) \ge 1.
		$$
		This yields \eqref{eq-subset} whenever $m\in M\setminus \widetilde{M}$.
		
		We now prove  \eqref{eq-m} by contradiction. Suppose, for the sake of contradiction, that \eqref{eq-m} fails. Then 
		there exists a sequence  $\left \{ p_k \right \}_k$ with $p_k\neq 0$,  such that
		\begin{align*}
			\sup_{M'}|p_k|< \frac{1}{k} \left \| p_k \right \|_{L^\infty(\mathbb{S}^{n-1})}.
		\end{align*}
		where $M'=M\setminus \widetilde{M}$. 
		Set 
		\begin{align*}
			q_k:= \frac{p_k}{\left \| p_k \right \|_{L^\infty(\mathbb{S}^{n-1})}},
		\end{align*}
		then we have
		\begin{align}\label{eq-unit-function}
			\left \| q_k \right \|_{L^\infty(\mathbb{S}^{n-1})}=1,\qquad \sup_{M'}|q_k|< \frac{1}{k}.
		\end{align}
		Let $P_N(\mathbb{S}^{n-1})$ denote the space of all functions \( p=P|_{\mathbb{S}^{n-1}}\) has degree at most  \( N\in \mathbb{Z}^+\).
		This space is finite-dimensional, and hence the unit sphere in the \( L^\infty(\mathbb{S}^{n-1}) \)-norm is compact. Consequently, the sequence \( \{q_k\} \) admits a convergent subsequence, still denoted by \( q_k \),  whose limit we denote by \( q^* \in P_N(\mathbb{S}^{n-1}) \),  satisfying
		\begin{align*}
			\left \| q^* \right \|_{L^\infty(\mathbb{S}^{n-1})}=1, \ \ \ \ \lim_{k \to \infty} q_k = q^*.
		\end{align*}
		Furthermore, for every \( x \in M' \), we have
		\begin{align}\label{eq-poly-q}
			|q^*(x)|=\lim_{k \to \infty} |q_k(x)|\le \lim_{k \to \infty}\frac{1}{k}=0,
		\end{align}
		which implies $q^*|_{M'}=0$.
		
		Now observe that, by our assumption $C_\mathcal{H}^{n-2+\delta}(M)>0$ and the fact that $C_\mathcal{H}^{n-2+\delta}(\widetilde{M})=0$, it follows that
		\begin{align*}
			C_\mathcal{H}^{n-2+\delta}(M')=C_\mathcal{H}^{n-2+\delta}(M\setminus \widetilde{M})>0,
		\end{align*}
		and therefore
		\begin{align*}
			\dim_H M'\ge n-2+\delta.
		\end{align*}
		However, this contradicts to the well-known fact that 
		\begin{align}\label{eq-zero-set}
			\dim_H\left \{ x\in \mathbb{S}^{n-1} \big| p(x)=0  \right \} \le n-2
		\end{align}
		for any nontrivial spherical polynomial. Therefore, \eqref{eq-m} holds.
		
		Now we are in the position to prove \eqref{eq-p-W} and \eqref{W-content}. On the one hand, by \eqref{eq-subset}, there exists a sequence $\left \{ r_k \right \}_k$ satisfying $\lim_{k \to +\infty}  r_k =0$ such that 
		\begin{align*}
			\lim_{k \to +\infty} r_k^{-n+2-\delta}C_\mathcal{H}^{n-2+\delta}(M \cap K(m,r_k))\ge 1.
		\end{align*}
		Moreover, there exists $N\in \mathbb{N}$ such that when $k>N$, 
		\begin{align}\label{eq-780}
			C_\mathcal{H}^{n-2+\delta}(M \cap K(m,r_k)) \ge r_k^{n-2+\delta}.
		\end{align}
		On the other hand, since $m\in M\setminus \widetilde{M}$ satisfies \eqref{eq-m}, $p$ is a spherical polynomial, there exists $\tilde{r} \in \left \{ r_k \right \}_{k\ge N}$ such that
		\begin{align}\label{eq-772}
			|p(x)|\ge \frac{c_0}{2}\left \| p \right \|_{L^\infty(\mathbb{S}^{n-1})}, \ \ \forall x\in M\cap K(m,\tilde{r}).
		\end{align}
		Combining \eqref{eq-780} and \eqref{eq-772}, there exists $\tilde{r} \in \left \{ r_k \right \}_{k\ge N}$ such that \eqref{eq-p-W} and \eqref{W-content} hold. Defining $W$ as \eqref{eq-W-01}, we complete the proof.
		
	\end{proof}
	\begin{remark}
		For later use in the proof of Lemma \ref{lemma-back}, we choose a constant $r_0>0$ such that
		\begin{align}\label{eq-r}
			\frac{r_0}{2}\left ( \frac{diam(M)}{diam(\mathbb{S}^{n-1})} \right )^{\frac{n-2+\frac{\delta }{2}}{n-2+\delta}} \le \tilde{r} < r_0.
		\end{align}
	\end{remark}
	
	Note that the subset \(W\)  in Lemma \ref{lemma-construct-W} is measured  in terms of  Hausdorff content, whereas the slicing Lemma \ref{lemma-slice} is formulated in terms of  Riesz capacity. In order to justify the applicability of the slicing theorem to $W$, we employ the following lemma, which quantifies the relationship between Hausdorff contents and Riesz capacities. This lemma has been established in \cite[p.28]{Car}, \cite[p.234]{KS}, see also \cite[Lemma 5.11]{HWW2025logtype}.
	
	\begin{lemma}\label{lemma-HWW}
		Let $s>0$, the following conclusions are true:
		
		(i) For every bounded Borel set $E\subset \mathbb{R}^n$,
		\begin{align}\label{eq-Cs-521}
			C_\mathcal{H}^{s}(E)\ge \frac 12 C_{s}(E).
		\end{align}
		
		(ii) Assume $s>s'>0$, then there exists a constant $A>0$, depending on $n,s,E$, such that when $E\subset \mathbb{R}^n$ is a compact set,
		\begin{align}\label{eq-Content-01}
			C_{s'}(E)\ge AC_\mathcal{H}^{s}(E).
		\end{align}
	\end{lemma} 
	
	Next, by applying the slicing lemma to $W$, we prove the existence of a spherical segment $\tilde{I}_a$ that satisfies certain specific properties. This segment will be essential for the proof  of Theorem \ref{THM-smallness on the sphere}. We begin by recalling the notion of a spherical line segment, which serves as the spherical analogue of a straight line segment in  $\mathbb{R}^n$. 
	\begin{definition}
		Given a starting point $p \in \mathbb{S}^{n-1}$ and a direction $v \in \mathbb{S}^{n-1}$ such that $p \cdot v = 0$, the spherical curve is defined as
		\begin{equation}\label{eq-map-520}
			\kappa = \kappa_v : [0, 2\pi] \to \mathbb{S}^{n-1}, \quad t \mapsto \cos(t)p + \sin(t)v.
		\end{equation}
		The spherical line segment starting at $p$ in the direction $v$ is 
		\begin{equation*}
			I = \operatorname{image}(\kappa|_{[0, l]}),\quad l>0.
		\end{equation*}
	\end{definition}
	
	\begin{lemma}\label{lemma-back} 
		Let $\delta\in (0,1)$. Suppose $M\subset \mathbb{S}^{n-1}$ with $C_\mathcal{H}^{n-2+\delta}(M)>0$ is a closed subset, then there exist a spherical line segment $\tilde{I}_a\subset \mathbb{S}^{n-1}$ and a constant $C_2=C_2(M,n,\delta)>0$ such that 
		\begin{align}\label{eq-line}
			C_{\frac{\delta}{2}}(M\cap \tilde{I}_a)>0, \quad \tilde{I}_a \cap W \neq \varnothing,
		\end{align}
		and 
		\begin{align}\label{eq-back}
			C_{n-2+\frac{\delta}{2}}(M)\le C_2 C_{\frac{\delta}{2}}(M \cap \tilde{I}_a).
		\end{align}
	\end{lemma}
	
	\begin{proof}
		By \eqref{eq-Content-01} and \eqref{W-content}, it follows that
		\begin{align}\label{eq-cap-W}
			C_{n-2+\frac{\delta}{2}}(W)\ge A(M,n,\delta) C_\mathcal{H}^{n-2+\delta}(W)>0,
		\end{align}
		then, applying Lemma \ref{lemma-slice} to $W$ with $s=n-2+\frac{\delta}{2}$, $m=n-2$, we derive that
		\begin{align*}
			c(n,m)C_{n-2+\frac{\delta}{2}}(W)\le \int \int_{V^\perp} C_{\frac{\delta}{2}}(W \cap V_a) \, d\mathcal{H}^{n-2} a \, d\gamma_{n,2} V,
		\end{align*}
		where $V\in G(n,2)$. Thus, there exists some two-dimensional plane $\tilde{V}_a$ such that
		\begin{align*}
			0<c(n,m)C_{n-2+\frac{\delta}{2}}(W)\le C_{\frac{\delta}{2}}(W \cap \tilde{V}_a),
		\end{align*}
		this yields
		\begin{align}\label{eq-845}
			C_{\frac{\delta}{2}}(M \cap \tilde{V}_a)\ge c(n,m)C_{n-2+\frac{\delta}{2}}(W)>0\ \ \text{and}\ \ W \cap \tilde{V}_a\neq \varnothing.
		\end{align}
		Observe that the intersection \(\tilde{\kappa}_a := \mathbb{S}^{n-1} \cap\tilde{V}_a\) is a spherical curve on $\mathbb{S}^{n-1}$. Choose a spherical cap $K$ such that $M\subset K$ and define the spherical line segment as
		\begin{align*}
			\tilde{I}_a = \tilde{\kappa}_a\cap K \subset K,
		\end{align*}
		then it follows from the fact $W\subset M\subset \mathbb{S}^{n-1}$ that 
		\begin{align*}
			M \cap \tilde{V}_a=M\cap \tilde{I}_a\ \text{ and }\ W \cap \tilde{V}_a=W\cap \tilde{I}_a,
		\end{align*}
		this, together with \eqref{eq-845}, implies \eqref{eq-line} and the following
		\begin{align}\label{eq-857}
			C_{\frac{\delta}{2}}(M \cap \tilde{I}_a)\ge c(n,m)C_{n-2+\frac{\delta}{2}}(W).
		\end{align}
		
		It remains to prove \eqref{eq-back}. From \eqref{eq-back} and \eqref{eq-857}, it suffices to establish a  quantitative relation between $C_{n-2+\frac{\delta}{2}}(W)$ and $C_{n-2+\frac{\delta}{2}}(M)$. Indeed, we have
		\begin{align}\label{eq-864}
			C_{n-2+\frac{\delta}{2}}(W) & \ge A(M,n,\delta) C_\mathcal{H}^{n-2+\delta}(W)\nonumber\\
			& \ge A(M,n,\delta)\cdot \tilde{r}^{n-2+\delta}\nonumber\\
			&> A(M,n,\delta)\cdot r_0^{n-2+\delta}\cdot\left ( \frac{diam(M)}{diam(\mathbb{S}^{n-1})}\right )^{n-2+\frac{\delta}{2}}\nonumber\\
			&\ge A(M,n,\delta)\cdot \frac{r_0^{n-2+\delta}}{(diam(\mathbb{S}^{n-1}))^{n-2+\frac{\delta}{2}}}\cdot C_{n-2+\frac{\delta}{2}}(M),
		\end{align}
		where the first inequality follows from \eqref{eq-cap-W},  the second from \eqref{W-content} and the third from \eqref{eq-r}.
		The last inequality is a consequence of the standard Riesz capacity estimate
		\begin{align*}
			C_s(B(a,r))\sim r^s,\ \ \  \forall a\in \mathbb{R}^n,r>0,0<s<n.
		\end{align*}
		Substituting \eqref{eq-864} into \eqref{eq-857} yields \eqref{eq-back}, thereby completing the proof of Lemma \ref{lemma-back}.
	\end{proof}
	
	\subsection{Reduction to spherical line segments}\,
	
	Theorem \ref{THM-smallness on the sphere} follows immediately by combining Lemma  \ref{lemma-back} and the following
	
	
	
	
	\begin{lemma}\label{lemma-spherically}
		Let $\delta \in (0,1)$. Let $p$ be a spherical polynomial of degree at most $N\in \mathbb{Z}^+$. Suppose $M\subset\mathbb{S}^{n-1}$ with $C_\mathcal{H}^{n-2+\delta}(M)>0$ is a closed subset, then for every spherical line segment $I \subset \mathbb{S}^{n-1}$ satisfying
		\begin{align}\label{eq-I-condition}
			C_{\frac{\delta}{2}}(M \cap I) > 0 \quad \text{and} \quad I \cap W \neq \varnothing,
		\end{align}
		there exists an absolute constant $C_0 > 0$ such that
		\begin{align}\label{eq-remez-spherically}
			\sup_{\mathbb{S}^{n-1}} |p|\le \frac{2}{c_0} \left ( \left ( \frac{C_0\cdot C_1(\delta) \cdot 2N}{C_{\frac{\delta}{2}}(M\cap I)}\right )^{\frac{2}{\delta}} {(2N)}^{\frac{2}{\delta}-1} \right )^{2N}\sup_{M\cap I} |p|, \  \  \  \delta \in(0,1),
		\end{align}
		where $c_0>0$ is a constant defined in \eqref{eq-p-W}.
	\end{lemma}
	
	\noindent {\it{Proof of Theorem \ref{THM-smallness on the sphere}.} }
	By Lemma \ref{lemma-back} and Lemma \ref{lemma-spherically}, we have
	\begin{align}\label{eq-2.48}
		\sup_{\mathbb{S}^{n-1}} |p|\le \frac{2}{c_0} \left ( \left ( \frac{C_2N}{C_{n-2+\frac{\delta}{2}}(M)}\right )^{\frac{2}{\delta}} N^{\frac{2}{\delta}-1} \right )^{2N}\sup_{M\cap I} |p|.
	\end{align}
	Furthermore, Lemma \ref{lemma-HWW} yields
	\begin{align*}
		C_{n-2+\frac{\delta}{2}}(M)\ge A(M,n,\delta) C_\mathcal{H}^{n-2+\delta}(M)>0,
	\end{align*}
	substituting this inequality into \eqref{eq-2.48} gives \eqref{eq-remez-sphere-whole}. This completes the proof of Theorem \ref{THM-smallness on the sphere}. \qed
	
	\noindent {\it{Proof of Lemma \ref{lemma-spherically}.} } 
	We divide the proof of \eqref{eq-remez-spherically} into three steps.
	
	{\it{Step 1: Existence of $I$ satisfying \eqref{eq-I-condition}.}} 
	It follows from Lemma \ref{lemma-HWW} (ii) that
	\begin{align*}
		0<C_\mathcal{H}^{n-2+\delta}(M)\lesssim_{M,n,\delta}C_{n-2+\frac{\delta}{2}}(M).
	\end{align*}
	Then, applying Lemma \ref{lemma-Ma-10.8} to $M$ with parameters $s=n-2+\frac{\delta}{2}$, $m=n-2$ yields that for \( \gamma_{n,2} \)-almost all \( V \in G(n,2) \),
	\begin{align*}
		\mathcal{H}^{n-2}(A)=\mathcal{H}^{n-2} \bigl( \{ a \in V^\perp : C_{\frac{\delta}{2}}(M \cap V_a) > 0 \} \bigr) > 0.
	\end{align*}
	For any $a\in A$, \(V_a\) is a two-dimensional plane, thus the intersection \(I_a := \mathbb{S}^{n-1} \cap V_a\) is a spherical line segment on $\mathbb{S}^{n-1}$. 
	Then, it follows that
	$$M\cap V_a=M\cap I_a,$$
	moreover, 
	\begin{align*}
		\mathcal{H}^{n-2} \bigl( \{ a \in V^\perp : C_{\frac{\delta}{2}}(M \cap I_a) > 0 \} \bigr) > 0.
	\end{align*}
	Note that $W\subset M$ satisfies  $C_{n-2+\frac{\delta}{2}}(W)>0$, thus there exists some $a\in A$ such that $C_{\frac{\delta}{2}}(M\cap I_a)=C_{\frac{\delta}{2}}(M\cap V_a)>0$ and $I_a\cap W\neq \varnothing$.
	
	{\it{Step 2: Reduction to a spherical line segment.}} 
	Let $I\subset \mathbb{S}^{n-1}$ be a spherical line segment satisfying
	\begin{align*}
		C_{\frac{\delta}{2}}(M \cap I) > 0 \quad \text{and} \quad I \cap W \neq \varnothing.
	\end{align*}
	Pick a point  $w\in I \cap W$. Denote by  $v$ the direction vector of $I$  such that $w\cdot v=0$, then the curve 
	\begin{align*}
		\kappa(t) =cos(\sigma(I)t)w+ sin(\sigma(I)t)v,\ \ \ t\in [0,1],
	\end{align*}
	parameterizes $I$, where $\sigma(I)$ denotes the arc length of $I$.
	
	On the one hand, since $w\in I$, there exists $\alpha \in [0,1]$ such that $w=\kappa(\alpha)$.  Moreover,
	\begin{align}\label{eq-K}
		\left \| p \circ \kappa \right \|_{L^{\infty}([0,1])}\ge |(p \circ \kappa)(a)| =|p(w)|\ge \frac{c_0}{2}\left \| p \right \|_{L^\infty(\mathbb{S}^{n-1})},
	\end{align}
	where the last inequality follows from the fact that  $w\in W$ and the property \eqref{eq-p-W}.
	
	On the other hand, set
	\begin{align}\label{eq-set-A-520}
		A:= \{t \in [0, 1] : \kappa(t) \in M\},
	\end{align}
	i.e., $A=\kappa^{-1}(M\cap I).$ 
	Consequently,
	\begin{align}\label{eq-A}
		\sup_{M \cap I} |p| = \|p \circ \kappa\|_{L^\infty(A)}.
	\end{align}
	
	In view of \eqref{eq-K} and \eqref{eq-A}, we see that to prove \eqref{eq-remez-spherically}, it suffices to bound $\left \| p \circ \kappa \right \|_{L^{\infty}([0,1])}$ by $\|p \circ \kappa\|_{L^\infty(A)}$, where $p \circ \kappa$ is a one-dimensional function defined on the spherical line segment $I$. 
	
	{\it{Step 3: Application of Tur\'{a}n's inequality to  $(p \circ \kappa)(t)$.}}
	We first claim  that  there are finite sequences $\{  \beta_k \}\subset\mathbb{C} $ and integers $\{ \lambda_k \}$ such that 
	\begin{align}\label{eq-exp}
		(p \circ \kappa)(t) = \sum_{k=1}^{2N+1} \beta_k e^{2\pi i c \lambda_k t},
	\end{align}
	where 
	\begin{equation}\label{eq-poly-520}
		c=\frac{1}{2\pi}\sigma(I)\quad \text{and}\quad -N\le \lambda_1\le\cdots\le \lambda_{2N+1}\le N.
	\end{equation}
	Indeed, it follows from \eqref{eq-map-520} that there exist multi-indices $w$, $v$, $\alpha$ such that
	\begin{align*}
		(p \circ \kappa)(t) = \sum_{|\alpha| \leq N} c_{\alpha} \prod_{j=1}^{n} \bigl[w_j \cos(\sigma(I)t) + v_j \sin(\sigma(I)t)\bigr]^{\alpha_j}.
	\end{align*}
	Note that
	\[
	w_j \cos(\sigma(I)t) + v_j \sin(\sigma(I)t) = w_j \Re(e^{i\sigma(I)t}) + v_j \Im(e^{i\sigma(I)t}),
	\]
	thus \eqref{eq-exp} and \eqref{eq-poly-520} follow from the binomial theorem.
	
	To proceed, we need the following Tur\'{a}n type inequality on fractal sets, which was recently established in \cite{YH}:
	{\textit{Let $\alpha\in(0,1)$ be given. For any $E \subset [0, 1]$ with $C_{\mathcal{H}}^{\alpha}(E) > 0$, and for every trigonometric polynomial 
			\begin{align*}
				\bar{p}(t)=\sum_{k=1}^q c_k e^{2\pi i m_k t}, \quad c_k \in \mathbb{C},
			\end{align*}
			with distinct frequencies
			\[
			m_1 < \cdots < m_q, \quad m_k \in \mathbb{Z},\ 1 \leq k \leq q.
			\]
			There exists some absolute constant $C_0 > 0$ such that
			\begin{align}\label{eq-Turan-YH}
				\sup_{t \in [0, 1]} |\bar{p}(t)| \leq \left( \left( \frac{C_0(q-1)}{C_{\mathcal{H}}^{\alpha}(E)} \right)^{1/\alpha} (m_q - m_1)^{1/\alpha-1} \right)^{q-1} \sup_{t \in E} |\bar{p}(t)|.
	\end{align}}}
	
	To use the above result, we shall further verify that
	\begin{align}\label{eq-2.42}
		C_{\mathcal{H}}^{\frac{\delta}{2}}(A)>0,
	\end{align}
	where $A$ is given by \eqref{eq-set-A-520}.
	Indeed, we have
	\begin{align}\label{eq-positive}
		C_{\frac{\delta}{2}}(A)=C_{\frac{\delta}{2}}(\kappa^{-1}(M\cap I)) \gtrsim_\delta C_{\frac{\delta}{2}}[(\kappa\circ\kappa^{-1})(M\cap I)]= C_{\frac{\delta}{2}}(M\cap I)>0,
	\end{align}
	where we used the fact that $\kappa(t)$ is a Lipschitz function and the known fact that if \( f: A \to \mathbb{R}^m \), where \( A \subset \mathbb{R}^n \), is a Lipschitz map with the constant $\operatorname{Lip}(f)$, then (see \cite[Theorem 9.1]{Ma})
	\[
	C_s(fA) \leq \operatorname{Lip}(f)^s \, C_s(A) \quad \text{for } s > 0.
	\]
	Combining \eqref{eq-positive} with \eqref{eq-Cs-521} in Lemma \ref{lemma-HWW} yields \eqref{eq-2.42}.
	
	
	
	Now we apply \eqref{eq-Turan-YH} with $\bar{p}=p \circ \kappa$, $E=A$, $|m_q-m_1|=2N$, and obtain
	\begin{align*}
		\left \| p \circ \kappa \right \|_{L^{\infty}([0,1])}\le \left ( \left ( \frac{C_0 \cdot 2N}{C_{\mathcal{H}}^{\frac{\delta}{2}}(A)}\right )^{\frac{2}{\delta}} (2N)^{\frac{2}{\delta}-1} \right )^{2N} \|p \circ \kappa\|_{L^\infty(A)}.
	\end{align*}
	This, together with \eqref{eq-K}, \eqref{eq-A}, \eqref{eq-Cs-521},  and \eqref{eq-positive}, yields \eqref{eq-remez-spherically}. Therefore, the proof is complete.
	\qed
	
	

	\begin{remark}\label{remark-construct-W}
		We point out that a direct slicing argument applied to $M$ gives a spherical line segment \( I_b \) with \( C_{\delta/2}(M \cap I_b) > 0 \), but it may not intersect \( W \).
		The requirement \( I_b \cap W \neq \emptyset \) is essential for the later estimate  \eqref{eq-K}, which motivates the construction of \( W \) and the choice \( I = \tilde{I}_a \) in Lemma \ref{lemma-spherically}.
		
	\end{remark}
	
	\section{Applications to observability inequalities}\label{sec3}
	
	In this section we apply Theorem \ref{THM-smallness on the sphere} to prove Theorems \ref{THM-observability inequality2} and \ref{THM-observability inequality3}. The proofs are presented in Subsections \ref{sec-spherical-heat} and \ref{sec-2m}, respectively.
	
	\subsection{Proof of Theorem \ref{THM-observability inequality2}}\label{sec-spherical-heat}\,
	
	The key ingredient  is to establish the following \emph{spectral inequality} associated for $\Delta_{\mathbb{S}^{n-1}}$: Let  $f \in \operatorname{Ran} P_{-\Delta_{\mathbb{S}^{n-1}} }((-\infty, \lambda])$, where $\lambda>0$, then there exist constants $d_0, d_1, d_2>0$  (depending on $n$, $\delta$, $M$)  
	such that
	\begin{align}\label{eq-spec-sphere}
		\left \| f \right \|_{L^2(\mathbb{S}^{n-1})}\le d_0 \cdot \exp\left( 
		d_1 \cdot \sqrt{\lambda} 
		+ d_2 \cdot \sqrt{\lambda} \cdot \log \lambda 
		\right)\sup_{M} |f|.
	\end{align}
	
	To prove \eqref{eq-spec-sphere}, we first recall that the eigenfunctions of $\Delta_{\mathbb{S}^{n-1}} $ are given by the spherical harmonics \(Y_{\ell, k}\), where \(\ell \in \mathbb{N}\) denotes the degree and \(k = 1, \ldots, n_\ell\) the multiplicity. More precisely, we have
	\begin{align}\label{eq-eigenvalue}
		-\Delta_{\mathbb{S}^{n-1}} Y_{\ell, k} = \lambda_\ell Y_{\ell, k}, \qquad \lambda_\ell = \ell(\ell + n - 2),
	\end{align}
	for all \(\ell \in \mathbb{N}\) and \(k = 1, \dots, n_\ell\). Moreover,
	the system \((Y_{\ell, k})_{\ell, k}\) forms an orthonormal basis of \(L^2(\mathbb{S}^{n-1})\).
	
	If  $f \in \operatorname{Ran} P_{-\Delta_{\mathbb{S}^{n-1}} }((-\infty, \lambda])$, then by \eqref{eq-eigenvalue}, $f$ is a finite linear combination of $Y_{\ell, k}$ with $\ell \leq \lambda^{1/2}$.
	Applying \eqref{eq-remez-sphere-whole} in Theorem \ref{THM-smallness on the sphere} with $N=\lambda^{1/2}$ and $p=f$, we deduce 
	\begin{align*}
		\left \| f \right \|_{L^2(\mathbb{S}^{n-1})}\le C_{spec} \sup_{M} |f|,
	\end{align*}
	with $C_{spec}$ given by
	\begin{align*}
		C_{spec} = d_0 \cdot \exp\left( 
		d_1 \cdot \sqrt{\lambda} 
		+ d_2 \cdot \sqrt{\lambda} \cdot \log \lambda 
		\right),
	\end{align*}
	where
	\begin{align*}
		d_0 = |\mathbb{S}^{n-1}|^{\frac12}\cdot\frac{2}{c_0}, \quad 
		d_1 = \frac{4}{\delta}\cdot \log \left(\frac{C(M,n,\delta) }{C_\mathcal{H}^{n-2+\delta}(M)}\right), \quad 
		d_2 = \frac{4-\delta}{\delta}.
	\end{align*}
	Then \eqref{eq-spec-sphere} follows.
	
	Once the spectral  inequality  \eqref{eq-spec-sphere}  is established,  the observability inequality \eqref{eq-obs-sphere} follows by applying the same arguments as in  \cite[Section 4]{Burq}, which in turn rely on the adapted Lebeau–Robiano strategy from   \cite{cone}. We therefore omit the details, and the proof of Theorem \ref{THM-observability inequality2} is complete.
	\qed
	

	As a consequence, Theorem \ref{THM-observability inequality2} implies  the null-controllability for the  heat equation on the sphere (see \cite[Section 5]{Burq}). Let $\mathbb{M}(\mathbb{S}^{n-1})$ denote the space of all Borel measures on $\mathbb{S}^{n-1}$. 
	
	\begin{corollary}\label{controllability1}
		Let $T>0$. Then for each $u_0\in L^2(\mathbb{S}^{n-1})$, there is a control 
		$$\mu \in L^\infty((0,T); \mathbb{M}(\mathbb{S}^{n-1})),\quad supp\ \mu \subset [0,T]\times M,$$
		with
		$$\left \| \mu \right \|_{L^\infty((0,T); \mathbb{M}(\mathbb{S}^{n-1}))}\le C_{obs}\left \| u_0 \right \|_{L^2(\mathbb{S}^{n-1})},$$ 
		such that the solution $u$ of the controlled equation
			\begin{align*}
				\begin{cases} 
					\partial_t u(t, x)-\Delta_{\mathbb{S}^{n-1}} u(t, x)=\mu (t, x), & t>0,\; x\in\mathbb{S}^{n-1}, \\  
					u(0, \cdot)=u_0\in L^2(\mathbb{S}^{n-1}),
				\end{cases}
			\end{align*}
			satisfies $u(T)=0$ over $\mathbb{S}^{n-1}$.
		\end{corollary}
		
		\subsection{Proof of Theorem \ref{THM-observability inequality3}}\label{sec-2m}\,
		
		The main step is to establish the  following \emph{spectral inequality} for the operator
		$H := -\Delta + |x|^{2m}$ ($m \in \mathbb{Z}^+,m\ge 2$).
		\begin{lemma}\label{lemma-spectral}
			Let $\Gamma_1$ and $\Theta_1$ be given by \eqref{eq-UN-har-526-0} and \eqref{eq-UN-har-526-1}, respectively.
			If $v \in \operatorname{Ran} P_{H}((-\infty, \mu^2])$ with $\mu>0$, then there exists a constant $C>0$, depending on $\Theta_1,n,\delta,m$, 
			such that
			\begin{align}\label{eq-spectral}
				\int_{\mathbb{R}^n} |v(x)|^2 dx \leq C e^{C \mu^{1+1/m}(1+|\log \mu|)} \int_{r_0}^{r_1} \left(\sup_{\omega\in \Theta_1} |v(r,\omega)|\right)^2 dr.
			\end{align}
		\end{lemma}
		
		\begin{proof}
			Let 
			\begin{equation*}
				0<\mu_1<\cdots<\mu_k\longrightarrow\infty,
			\end{equation*}
			be the distinct eigenvalues for the operator $H := -\Delta + |x|^{2m}$. 
			Since the potential is radial, the standard separation of variables arguments yield that for each $v \in \mathbf{1}_{H<\mu^2}L^2(\mathbb{R}^n)$, it admits the expansion (see, e.g., in \cite[pp. 188-189]{BS} or \cite[p. 90]{Reed})
			\begin{align}\label{eq-eigenfunction}
				v=\sum_{\ell=0}^{\ell_{max}}r^{-\frac{n-1}{2}}f(r)Y_{\ell, k_\ell}(\omega), \quad \omega\in \mathbb{S}^{n-1}, 
			\end{align}
			where \( Y_{\ell, k_\ell} \) denote the  
			spherical harmonics, \(\int_0^\infty |f(r)|^2 \, dr < \infty\), and $\ell_{max}\in \mathbb{N}$ is determined by $\mu$. 
			Moreover, for each eigenvalue $\mu_k$, the radial function satisfies
			\begin{align}\label{eq-eigenequation-2m}
				\frac{d^2f(r)}{dr^2}+\left \{ \mu_k-r^{2m}- \frac{\ell(\ell+n-2)+\frac{(n-1)(n-3)}{4}}{r^2}\right \} f(r)=0
			\end{align}
			for certain values of $\ell\in \mathbb{N}$, where $r\in (0, +\infty)$.
			The key ingredient is to establish an upper bound for $\ell_{max}$, which enables the application of  Theorem \ref{THM-smallness on the sphere}.  This is achieved through the following claim.
			
			\noindent \textbf{Claim:} Let $\mu_k, \ell$ satisfy the equation \eqref{eq-eigenequation-2m}, then
			\begin{align}\label{eq-eigen-2m}
				\mu_k \underset{}{\sim} c_m \left ( \frac{1}{2}\left (\ell+\frac{n-3}{2}  \right ) + k + \frac{3}{4} \right )^{\frac{2m}{m+1}},\quad k \to +\infty,
			\end{align}
			where 
			$$c_m=\left ( 4m\sqrt{\pi}\cdot\frac{\Gamma(\frac{1}{2m}+\frac{3}{2})}{\Gamma(\frac{1}{2m})} \right )^{\frac{2m}{m+1}}.$$
			We point out that in three dimension, \eqref{eq-eigen-2m} was established in \cite{Titchmarsh1}. The higher-dimensional case follows by a similar argument; we provide a proof for $n>3$ in the appendix (see Section \ref{appendix}) for the sake of completeness.
			
			By \eqref{eq-eigen-2m} and the assumption $\mu_k\le \mu ^2$, we obtain that
			\begin{align*}
				\ell_{max} \le 2 (c_m^{-1})^{\frac{m+1}{2m}} \mu ^{1+\frac{1}{m}}.
			\end{align*}
			Applying the Remez type inequality \eqref{eq-remez-sphere-whole} with $N=2 (c_m^{-1})^{\frac{m+1}{2m}} \mu ^{1+\frac{1}{m}}$, we derive that 
			\begin{align}\label{eq-harmonics}
				\sup_{\omega\in \mathbb{S}^{n-1}}\left|\sum_{\ell=1}^{\ell_{max}}r^{-\frac{n-1}{2}}f(r)Y_{\ell, k_\ell}(\omega)\right|\le Ce^{C \mu^{1+1/m}(1+\log \mu)} \sup_{\omega\in \Theta_1}\left|\sum_{\ell}^{\ell_{max}}r^{-\frac{n-1}{2}}f(r)Y_{\ell, k_\ell}(\omega)\right|,
			\end{align}
			where $C>0$ depends on $\Theta_1,n,\delta,m$.
			Set
			\begin{align*}
				\Gamma_2 :=\{ x \in \mathbb{R}^n: r_0\le |x|\le r_1, x/|x| \in \mathbb{S}^{n-1} \}.
			\end{align*}
			Clearly, we have $|\Gamma_2|>0$, then it follows from \cite[Theorem 2.1]{Martin} that for each $v \in \mathbf{1}_{H<\mu^2}L^2(\mathbb{R}^n)$,
			\begin{align}\label{eq-Martin-2}
				\int_{\mathbb{R}^n} |v(x)|^2 \, dx &\leq C e^{C\mu^{1+1/m}|\log \mu|} \int_{\Gamma_2} |v(x)|^2 \, dx\nonumber\\
				&=C e^{C\mu^{1+1/m}|\log \mu|} \int_{r_0}^{r_1}\int_{\mathbb{S}^{n-1} }^{} |v(r,\omega)|^2 drd\omega.
			\end{align}
			Combining  \eqref{eq-eigenfunction}, \eqref{eq-harmonics} and \eqref{eq-Martin-2}, we have
			\begin{align}\label{eq-spec-unhar-527-1}
				\int_{\mathbb{R}^n} |v(x)|^2 \, dx \leq& C e^{C \mu^{1+1/m}(1+|\log \mu|)}\int_{r_0}^{r_1} \left(\sup_{\omega\in \Theta_1}\left|\sum_{\ell}^{\ell_{max}}r^{-\frac{n-1}{2}}f(r)Y_{\ell, k_\ell}(\omega)\right|\right)^2 dr\nonumber\\
				= & C e^{C \mu^{1+1/m}(1+|\log \mu|)} \int_{r_0}^{r_1} \left(\sup_{\omega\in \Theta_1} |v(r,\omega)|\right)^2 dr.
			\end{align}
			Therefore, the proof of Lemma \ref{lemma-spectral} is complete.
		\end{proof}

		Once the spectral inequality \eqref{eq-spectral} is established, the proof of Theorem \ref{THM-observability inequality3} follows from the adapted Lebeau-Robbiano strategy in \cite{cone}. For the reader’s convenience, we show how  \eqref{eq-spectral} is applied to obtain the following recurrence inequality: there exist $T_0 > 0$ and $\lambda_1 > 0$, such that for all $\lambda \ge \lambda_1$, $\tau \in (0, T_0]$ and $u_0 \in L^2(\mathbb{R}^n)$,
		\begin{align}\label{eq-1245}
			&f(\lambda) \bigl\| e^{-\tau(\lambda)H} u_0 \bigr\|_{L^2(\mathbb{R}^n)} - f\!\left(\tfrac{5}{4}\lambda\right) \|u_0\|_{L^2(\mathbb{R}^n)} \le \int_0^{\tau(\lambda)} \left( \int_{r_0}^{r_1} \Bigl( \sup_{\omega \in \Theta_1} \bigl| e^{-tH} u_0(r,\omega) \bigr| \Bigr)^2 dr \right)^{\frac{1}{2}} dt,
		\end{align}
		where $f$ is defined by \eqref{equ-1023-6}. The observability inequality \eqref{eq-obs-ineq-2m}  then follows from \eqref{eq-1245} by a standard telescoping series argument, which is exactly the same as in \cite[Section 4.3]{HWW2025logtype}, therefore, we omit the details.
		
		
		To prove \eqref{eq-1245}, we first define the orthogonal projection
		\begin{align*}
			\pi_\lambda: L^2(\mathbb{R}^n)\to \mathcal {E}_\lambda:= \operatorname{Ran} P_{H}((-\infty, \lambda]).
		\end{align*}
		By \eqref{eq-spectral}, there exist $C>0$ and a numerical constant $\lambda_0\gg 1$ such that
		\begin{align*}
			\left \| e^{-tH}\pi_\lambda u_0 \right \|_{L^2(\mathbb{R}^n)} &\le e^{C \lambda^{\frac{1}{2}+\frac{1}{2m}}(1+|\log \lambda|)} \left(\int_{r_0}^{r_1} \left(\sup_{\omega\in \Theta_1} |e^{-tH}\pi_\lambda u_0(r,\omega)|\right)^2 dr\right )^{\frac{1}{2}}\nonumber\\
			&\le e^{C \lambda(\log \lambda)^{-2}} \left(\int_{r_0}^{r_1} \left(\sup_{\omega\in \Theta_1} |e^{-tH}\pi_\lambda u_0(r,\omega)|\right)^2 dr\right )^{\frac{1}{2}},
		\end{align*}
		holds for $\lambda\ge \lambda_0$ and $u_0 \in L^2(\mathbb{R}^n)$.
		Then
		\begin{align*}
			\|e^{-\tau H}\pi_\lambda u_0\|_{L^2(\mathbb{R}^n)}&\le \frac{2}{\tau }\int_{\frac{\tau }{2}}^\tau \|e^{-t H}\pi_\lambda u_0\|_{L^2(\mathbb{R}^n)}dt\nonumber\\
			&\le \frac{2}{\tau }e^{C \lambda(\log \lambda)^{-2}}\int_{\frac{\tau }{2}}^\tau \left(\int_{r_0}^{r_1} \left(\sup_{\omega\in \Theta_1} |e^{-tH}\pi_\lambda u_0(r,\omega)|\right)^2 dr\right )^{\frac{1}{2}} dt.
		\end{align*}
		
		Next, we define two functions $\varphi$ and $\psi$ as follows:
		\begin{align*}
			\psi(t):=t\log t,\;\;t>0;\;\;\;\;\varphi(\lambda):=(\log \lambda)^2,\;\;\lambda\gg 1.
		\end{align*}
		Since
		\begin{align*}
			\lim_{\tau\to 0^+}\tau\psi(\frac{1}{\tau})=+\infty\;\;\mbox{and}\;\;\lim_{\lambda\to+\infty}\frac{C\lambda}{\varphi(\lambda)}=+\infty,
		\end{align*}
		there exists $T_0>0$ such that when $\tau\in (0,T_0]$, $\lambda\ge \lambda_0$,
		$$
		\frac{1}{\tau }e^{C \lambda(\log \lambda)^{-2}}=e^{\tau\psi(\frac{1}{\tau})+\frac{C\lambda}{\varphi(\lambda)}}\leq e^{\tau\psi(\frac{1}{\tau})\frac{C\lambda}{\varphi(\lambda)}},
		$$
		furthermore,
		\begin{align*}
			\|e^{-\tau H}\pi_\lambda u_0\|_{L^2(\mathbb{R}^n)}\le 2 e^{\tau\psi(\frac{1}{\tau})\frac{C\lambda}{\varphi(\lambda)}} \int_{\frac{\tau }{2}}^\tau \left(\int_{r_0}^{r_1} \left(\sup_{\omega\in \Theta_1} |e^{-tH}\pi_\lambda u_0(r,\omega)|\right)^2 dr\right )^{\frac{1}{2}} dt.
		\end{align*}
		It follows from the triangle inequality that when $\tau\in (0,T_0]$, $\lambda\ge \lambda_0$,
		\begin{align}\label{eq-1282}
			\left \| e^{-\tau H}u_0 \right \|_{L^2(\mathbb{R}^n)} &\le  \left \| e^{-\tau H}\pi_\lambda u_0 \right \|_{L^2(\mathbb{R}^n)}+\left \| e^{-\tau H}(1-\pi_\lambda)u_0 \right \|_{L^2(\mathbb{R}^n)}\nonumber\\
			&\le 2 e^{\tau\psi(\frac{1}{\tau})\frac{C\lambda}{\varphi(\lambda)}} \int_{\frac{\tau }{2}}^\tau \left(\int_{r_0}^{r_1} \left(\sup_{\omega\in \Theta_1} |e^{-tH}\pi_\lambda u_0(r,\omega)|\right)^2 dr\right )^{\frac{1}{2}} dt\nonumber\\
			&\ \ +\left \| e^{-\tau H}(1-\pi_\lambda)u_0 \right \|_{L^2(\mathbb{R}^n)}\nonumber\\
			&\le 2 e^{\tau\psi(\frac{1}{\tau})\frac{C\lambda}{\varphi(\lambda)}} \int_{\frac{\tau }{2}}^\tau \left(\int_{r_0}^{r_1} \left(\sup_{\omega\in \Theta_1} |e^{-tH}u_0(r,\omega)|\right)^2 dr\right )^{\frac{1}{2}} dt\nonumber\\
			&\ \ +2 e^{\tau\psi(\frac{1}{\tau})\frac{C\lambda}{\varphi(\lambda)}} \int_{\frac{\tau }{2}}^\tau \left(\int_{r_0}^{r_1} \left(\sup_{\omega\in \Theta_1} |e^{-tH}(1-\pi_\lambda) u_0(r,\omega)|\right)^2 dr\right )^{\frac{1}{2}} dt\nonumber\\
			&\ \ +\left \| e^{-\tau H}(1-\pi_\lambda)u_0 \right \|_{L^2(\mathbb{R}^n)}.
		\end{align}
		For the last term on the right hand side of \eqref{eq-1282}, it follows that
		\begin{align}\label{eq-1291}
			\left \| e^{-\tau H}(1-\pi_\lambda)u_0 \right \|_{L^2(\mathbb{R}^n)}\le e^{-\tau\lambda}\left \| u_0 \right \|_{L^2(\mathbb{R}^n)}.
		\end{align}
		For the second term on the right hand side of \eqref{eq-1282}, note that
		\begin{align*}
			&\int_{r_0}^{r_1} \left(\sup_{\omega\in \Theta_1} |e^{-tH}(1-\pi_\lambda) u_0(r,\omega)|\right)^2 dr\nonumber\\
			\le \ &\int_{r_0}^{r_1} \left(\sup_{\omega\in \mathbb{S}^{n-1}} |e^{-tH}(1-\pi_\lambda) u_0(r,\omega)|\right)^2 dr 
			\le  (r_1-r_0) \left( \sup_{\substack{\omega\in \mathbb{S}^{n-1} \\ r\in (r_0,r_1)}} \left| e^{-tH}(1-\pi_\lambda) u_0(r,\omega) \right| \right)^2\nonumber\\
			\le \ & (r_1-r_0) \left [ a_0(t-\frac{\tau}{3})^{-\frac{n}{4}}\left \| e^{-\frac{\tau}{3}H}(1-\pi_\lambda) u_0\right \|_{L^2((r_0,r_1)\times \mathbb{S}^{n-1} )}  \right ]^2,
		\end{align*}
		in the last inequality, we have used the standard $L^2\to L^\infty$ estimate:
		\begin{align*}
			\left \| e^{-tH} \right \|_{L^2((r_0,r_1)\times \mathbb{S}^{n-1})\to L^\infty((r_0,r_1)\times \mathbb{S}^{n-1})}\le a_0 t^{-\frac{n}{4}},\quad t>0.
		\end{align*}
		Thus
		\begin{align}\label{eq-1305}
			&\int_{\frac{\tau }{2}}^\tau \left(\int_{r_0}^{r_1} \left(\sup_{\omega\in \Theta_1} |e^{-tH}(1-\pi_\lambda) u_0(r,\omega)|\right)^2 dr\right )^{\frac{1}{2}} dt\nonumber\\
			\le \ & \int_{\frac{\tau }{2}}^\tau (r_1-r_0)^{\frac{1}{2}} a_0(t-\frac{\tau}{3})^{-\frac{n}{4}}\left \| e^{-\frac{\tau}{3}H}(1-\pi_\lambda) u_0\right \|_{L^2((r_0,r_1)\times \mathbb{S}^{n-1} )} dt\nonumber\\
			\le \ & (r_1-r_0)^{\frac{1}{2}} (\frac{6}{\tau})^{\frac{n}{4}} e^{-\frac{\lambda\tau}{3}}\left \| u_0 \right \|_{L^2(\mathbb{R}^n)}.
		\end{align}
		
		Letting $g(\tau,\lambda):=\frac{1}{2}e^{-C\tau\psi(\frac{1}{\tau})\frac{\lambda}{\varphi(\lambda)}}$, with $\tau\in(0,T_0]$ and $\lambda\geq\lambda_0$,  plugging \eqref{eq-1291} and \eqref{eq-1305} into \eqref{eq-1282}, we deduce that when $\tau\in(0,T_0]$ and $\lambda\geq \lambda_0$,
		\begin{align}\label{eq-1312}
			g(\tau,\lambda)\|e^{-\tau H}u_0\|_{L^2(\mathbb{R}^n)}
			&\leq  \int_{\frac{\tau }{2}}^\tau \left(\int_{r_0}^{r_1} \left(\sup_{\omega\in \Theta_1} |e^{-tH}u_0(r,\omega)|\right)^2 dr\right )^{\frac{1}{2}} dt \nonumber\\
			&\ +\left((r_1-r_0)^{\frac{1}{2}}a_0(\frac{6}{\tau})^{\frac{n}{4}}e^{-\frac{\lambda\tau}{3}}+g(\tau,\lambda)e^{-\lambda \tau}\right)\|u_0\|_{L^2(\mathbb{R}^n)}.
		\end{align}
		
		We now proceed to estimate the last term in \eqref{eq-1312}. Let $\tau$ be defined by
		\begin{align}\label{equ-1023-5}
			\frac{1}{\tau(\lambda)}:=\psi^{[-1]}\Big(\frac{\varphi(\lambda)}{4C} \Big), \quad \lambda\geq \lambda_0,
		\end{align}
		and then define $f$ by
		\begin{align}\label{equ-1023-6}
			f(\lambda):=g(\tau(\lambda),\lambda)=\frac{1}{2}e^{-\frac{\tau(\lambda)\lambda}{4}},\;\;\lambda\geq \lambda_0.
		\end{align}
		A direct calculation yields $\psi^{[-1]}(\lambda)\sim \frac{\lambda}{\log \lambda}$ for $\lambda\gg 1$. Consequently, from \eqref{equ-1023-5} we deduce
		\begin{align}\label{equ-1023-7}
			\tau(\lambda)\sim \frac{\log\log \lambda}{(\log \lambda)^2}, \quad \lambda\gg 1.
		\end{align}
		By \eqref{equ-1023-5} and
		\eqref{equ-1023-7},
		we see that for sufficiently small  $\varepsilon>0$, there is a constant $C_\varepsilon>0$ depending only on $\varepsilon>0$ such that
		\begin{align*}
			\tau(\lambda)^{-\frac{1}{4}}e^{-\frac{\lambda\tau(\lambda)}{3}}\leq C_\varepsilon e^{-(\frac{1}{3}-\varepsilon)\lambda \tau(\lambda)},\;\;\lambda\geq\lambda_0.
		\end{align*}
		This, along with \eqref{equ-1023-6} and the fact that $\tau(\lambda)\ge\tau(\frac{5}{4}\lambda)$ for $\lambda\gg 1$, yields that there is $\lambda_1>0$, depending only on $a_0$ and $L$, such that
		\begin{align}\label{equ-1023-8}
			\Big((r_1-r_0)^{\frac{1}{2}}a_0\Big(\frac{6}{\tau}\Big)^{\frac{1}{4}}e^{-\frac{\lambda\tau(\lambda)}{3}}+g(\tau(\lambda),\lambda)e^{-\lambda
				\tau(\lambda)}\Big)\|u_0\|_{L^2(\mathbb{R}^n)}\leq f\Big(\frac{5}{4}\lambda\Big)\|u_0\|_{L^2(\mathbb{R}^n)},\,\,\,\lambda\geq\lambda_1.
		\end{align}
		Therefore, the recurrence inequality \eqref{eq-1245} follows  by \eqref{eq-1312} and \eqref{equ-1023-8}.
		\qed
		
		\begin{remark}
			%
			Let $\left \{ \lambda_k \right \}_k$ denote the eigenvalues of $H= -\Delta + |x|^{2m}$ on $L^2(\mathbb{R}^n)$ counted with multiplicity. It is well known that (see \cite[Chapter 2]{Boggiatto})
			\begin{equation*}
				\lambda_k \underset{k \to +\infty}{\sim} c_{m} k^\frac{2m}{n(m+1)}.
			\end{equation*}
			This differs from the asymptotic behavior of $\mu_k$ in \eqref{eq-eigen-2m}. Nevertheless, the following identity holds  (see, e.g., \cite{Reed,Titchmarsh4}):
			\begin{align*}
				N(\lambda) = \sum_{\ell}^{} a_lN_\ell(\lambda),\quad a_l=\frac{(2\ell+n-2)(\ell+n-3)!}{\ell!(n-2)!},\quad n\ge 3,
			\end{align*}
			where \( N(\lambda) \) denotes the number of \( \lambda_k \) that do not exceed \( \lambda \), and \( N_\ell(\lambda) \) denotes the number of $\mu_k$ that do not exceed \( \lambda \) with angular momentum $\ell\in \mathbb{N}$. 
			
			
		\end{remark}

		\begin{appendix}\label{app-1}
			
			\section{Proof of \eqref{eq-eigen-2m}}\label{appendix}
			
			
			The goal of this section is to prove the asymptotic formula \eqref{eq-eigen-2m} in Subsection \ref{sec-2m}. In fact, we can prove a slightly more  general version. Assume that $q(r)$ is  radial  and satisfies
			
			(a) \( q(r)\in C^3(\mathbb{R}^+) \).
			Both  \( q(r) \) and \( q'(r) \) are non-decreasing functions of \( r \),
			
			(b)  $q'(r)/q(r)=q''(r)/q'(r) = q'''(r)/q''(r) =O(1/r)$, as \( r \to \infty \).
			
			\begin{proposition}\label{lemma-eigenvalue}
				Let $\mu_k$ and  $\ell$ ($k,\ell\in \mathbb{N}$) satisfy 
				\begin{align}\label{eq-eigen}
					\frac{d^2f(r)}{dr^2}+\left \{ \mu_k-q(r)- \frac{\ell(\ell+n-2)+\frac{(n-1)(n-3)}{4}}{r^2}\right \} f(r)=0,\quad 0<r<+\infty,
				\end{align}
				then
				\begin{align}\label{eq-eigenvalue-radial}
					\int_{0}^{R} \{\mu_k - q(r)\}^{\frac{1}{2}} \, dr = \left( \frac{1}{2}\left (\ell+\frac{n-3}{2}  \right ) + k + \frac{3}{4} \right) \pi + \delta_k,
				\end{align}
				where  $\lim_{k \to \infty} \delta_k = 0$, and \( R \) obeys \( q(R) = \mu_k \).    
			\end{proposition}

				
				
				
			
			In the particular case $q(r)=r^{2m}$ ($m\ge 2$), we have $R=\mu_k^{\frac{1}{2m}}$ and
			\begin{align*}
				\int_{0}^{R} \{\mu_k - q(r)\}^{\frac{1}{2}} \, dr =\frac{\sqrt{\pi}}{4m}\frac{\Gamma(\frac{1}{2m})}{\Gamma(\frac{1}{2m}+\frac{3}{2})}\mu_k^{\frac{m+1}{2m}},
			\end{align*}
			this, together with \eqref{eq-eigenvalue-radial}, yields \eqref{eq-eigen-2m}.
			
			{\textit{Proof of Proposition \ref{lemma-eigenvalue}.}} For $n=3$, \eqref{eq-eigenvalue-radial} was proved in~\cite{Titchmarsh1}. For odd $n>4$, we have
			\[
			\ell(\ell + n - 2) + \frac{(n - 1)(n - 3)}{4} = L(L + 1),
			\]
			where $L = \ell + \frac{n-3}{2} \in \mathbb{Z}^+$. Thus, \eqref{eq-eigenvalue-radial} follows directly from the $n=3$ case in \cite{Titchmarsh1}.
			
			
			It remains to prove \eqref{eq-eigenvalue-radial} when $n>4$ is an even integer. To this end, we adapt the proof in \cite{Titchmarsh1}, \cite[pp. 160-163]{Titchmarsh2}. The function
			\[
			F(r)=q(r)+\frac{\ell(\ell+n-2)+\frac{(n-1)(n-3)}{4}}{r^2},\ \ \ 0<r<+\infty,
			\]
			attains a unique positive minimum at some point \(r=a>0\), with
			\[
			\lim_{r\to0^+}F(r)=\lim_{r\to\infty}F(r)=+\infty.
			\]
			Moreover, \(F\) is strictly decreasing on \((0,a)\) and strictly increasing on \((a,\infty)\). Hence, we consider the intervals \((0,a)\) and \((a,\infty)\) separately.
			Let \(\psi(r,\mu_k)\) and \(\phi(r,\mu_k)\) be the solutions of \eqref{eq-eigen} belonging to \(L^2(a,\infty)\) and \(L^2(0,a)\), respectively. Then any solution $f(r)$ to \eqref{eq-eigen} must be a constant multiple of \( \phi(r, \mu_k) \) in \( (0, a) \) and of \( \psi(r, \mu_k) \) in \( (a, \infty) \); thus,
			\[
			\psi(r, \mu_k) = Af(r) \quad (a \leq r),
			\]
			\[
			\phi(r, \mu_k) = Bf(r) \quad (r \leq a),
			\]
			it follows that
			\begin{align}\label{eq-equation}
				\phi(a, \mu_k)\psi'(a, \mu_k) = ABf(a)f'(a) = \phi'(a, \mu_k)\psi(a, \mu_k).
			\end{align}
			
			
			To establish \eqref{eq-eigenvalue-radial}, we solve \eqref{eq-equation}. Following the arguments in \cite{Titchmarsh1,Titchmarsh2}, we derive  representations for  $\psi(a,\mu_k)$, $\psi'(a,\mu_k)$, $\phi(a,\mu_k)$, $\phi'(a,\mu_k)$, respectively. Unlike the settings in \cite{Titchmarsh1} and \cite{Titchmarsh2},  which treat the three dimensional case, here the Bessel function order  $\ell+\frac{n-2}{2}$ is an integer, so its asymptotic expansion takes a different form.  More precisely, by the Langer method (see \cite{Titchmarsh3}), we have
			\begin{align}\label{eq-psi}
				\psi(a, \mu_k) = 2e^{-\frac{2}{3}i\pi} &\left\{ \mu_k - q(a) - \frac{\ell(\ell+n-2)+\frac{(n-1)(n-3)}{4}}{a^2} \right\}^{-\frac 14} \nonumber\\
				&\cdot \left\{ \cos\left(z - \frac{1}{4}\pi\right) + O\left(\frac{1}{z}\right) \right\},
			\end{align}
			where
			\begin{align}\label{eq-1719}
				z = \int_a^T \left\{
				\mu_k - q(t) - \frac{\ell(\ell+n-2)+\frac{(n-1)(n-3)}{4}}{t^2}
				\right\}^{1/2} dt := F_a(\mu_k),
			\end{align}
			with the assumption that $\mu_k > q(a) + \frac{\ell(\ell+n-2)+\frac{(n-1)(n-3)}{4}}{a^2}$ and that $T>a$ is the zero of the integrand. Moreover,
			\begin{align}\label{eq-psi-diff}
				\psi'(a, \mu_k) = 2e^{-\frac{2}{3}i\pi} &\left\{ \mu_k - q(a) - \frac{\ell(\ell+n-2)+\frac{(n-1)(n-3)}{4}}{a^2} \right\}^{\frac{1}{4}} \nonumber\\
				&\cdot \left\{ \sin\left(z - \frac{1}{4}\pi\right) + O\left(\frac{1}{z}\right) \right\}.
			\end{align}
			For $\phi(a,\mu_k)$ and $\phi'(a,\mu_k)$, if $q(r)=0$, \eqref{eq-eigen} is a form of the Bessel equation of order \( \ell + \frac{n-2}{2} \). Hence \( \phi(r, \mu_k) \) satisfies the integral equation
			\begin{align}\label{eq-solution-bessel}
				\phi(r, \mu_k) &= r^{\frac{1}{2}} J_{\ell+\frac{n-2}{2}}(r\sqrt{\mu_k}) + \frac{2}{\pi} \int_0^r \Bigl\{J_{\ell+\frac{n-2}{2}}(r\sqrt{\mu_k}) Y_{\ell+\frac{n-2}{2}}(t\sqrt{\mu_k}) \nonumber\\
				&\ \ \ - J_{\ell+\frac{n-2}{2}}(t\sqrt{\mu_k}) Y_{\ell+\frac{n-2}{2}}(r\sqrt{\mu_k})\Bigr\} r^{\frac{1}{2}} t^{\frac{1}{2}} q(t)\phi(t, \mu_k) \, dt.
			\end{align}
			Subsequently, we first estimate the function $\phi(t, \mu_k)$ appearing in the integrand. Then, we estimate the integral in \eqref{eq-solution-bessel} by using the asymptotic expansion of Bessel functions \cite[p.133, p.242]{Asymptotics} and the estimate of $\phi(t, \mu_k)$. The details are as follows.
			
			Let
			\begin{align*}
				\omega(r, \mu_k) = 
				\begin{cases}
					r^{\ell+\frac{n-1}{2}} \mu_k^{\frac{\ell+\frac{n-2}{2}}{2}}, & r \leq \mu_k^{-\frac{1}{2}}, \\
					\mu_k^{-\frac{1}{4}}, & r > \mu_k^{-\frac{1}{2}}.
				\end{cases}
			\end{align*}
			Then 
			\begin{align*}
				|r^{\frac{1}{2}} J_{\ell+\frac{n-2}{2}}(r\sqrt{\mu_k})| < A\omega(r, \mu_k),
			\end{align*}
			where \( A \) denotes various positive constants. It follows that
			\begin{align}\label{eq-eq-1764}
				|\phi(r, \mu_k)| < A\omega(r, \mu_k),
			\end{align}
			if \( \mu_k \) is large enough. Indeed, let
			\[
			\max_{0 \leq r \leq a} \frac{|\phi(r, \mu_k)|}{\omega(r, \mu_k)} = M.
			\]
			Then, if \( r \leq \mu_k^{-\frac{1}{2}} \), \eqref{eq-solution-bessel} gives
			\[
			M < A + AM \int_0^r \left(\frac{r}{t}\right)^{\frac{\ell+1}{2}} r^{\frac{1}{2}} t^{\frac{1}{2}} \frac{\omega(t, \lambda)}{\omega(r, \mu_k)} \, dt = A + AMr^2,
			\] 
			while if \( r > \mu_k^{-\frac{1}{2}} \),
			\[
			M < A + AM \int_0^{\mu_k^{-\frac{1}{2}}} t \, dt + AM\mu_k^{-\frac{1}{2}} \int_{\mu_k^{-\frac{1}{2}}}^r \, dt < A + AM\mu_k^{-\frac{1}{2}}.
			\]
			Combining the above two estimates gives \eqref{eq-eq-1764}.
			
			Substituting \eqref{eq-eq-1764} with $r=t$ into the right-hand side of \eqref{eq-solution-bessel}, we obtain
			\begin{align}\label{eq-phi-small}
				\phi(r, \mu_k) = r^{\frac{1}{2}}J_{\ell+\frac{n-2}{2}}(r\sqrt{\mu_k}) + O(r^{\ell+2+\frac{n-1}{2}}\mu_k^{\frac{1}{2}(\ell+\frac{n-2}{2})}), \quad r < \mu_k^{-\frac{1}{2}},
			\end{align}
			and 
			\begin{align}\label{eq-phi-big}
				\phi(r, \mu_k) = r^{\frac{1}{2}}J_{\ell+\frac{1}{2}}(r\sqrt{\mu_k}) + O(\mu_k^{-\frac{3}{4}}), \quad r \ge \mu_k^{-\frac{1}{2}}.
			\end{align}
			
			Now we are in the position to estimate the integral in \eqref{eq-solution-bessel} and then derive the representations for $\phi(r, \mu_k)$ and $\phi'(r, \mu_k)$. Substituting \eqref{eq-phi-small} and \eqref{eq-phi-big} with $r=t$ into the right of \eqref{eq-solution-bessel}, we obtain if \( \mu_k^{-\frac{1}{2}} \leq r \leq a \),
			\begin{align}\label{eq-bessel-2}
				\phi(r, \mu_k) = r^{\frac{1}{2}} J_{\ell+\frac{n-2}{2}}(r\sqrt{\mu_k}) &+ \frac{2}{\pi} r^{\frac{1}{2}} J_{\ell+\frac{n-2}{2}}(r\sqrt{\mu_k}) \int_0^r J_{\ell+\frac{n-2}{2}}(t\sqrt{\mu_k}) Y_{\ell+\frac{n-2}{2}}(t\sqrt{\mu_k}) tq(t) \, dt\nonumber\\
				&- \frac{2}{\pi} r^{\frac{1}{2}} Y_{\ell+\frac{n-2}{2}}(r\sqrt{\mu_k}) \int_0^r J_{\ell+\frac{n-2}{2}}^2(t\sqrt{\mu_k}) tq(t) \, dt + O(\mu_k^{-\frac{5}{4}}).
			\end{align}
			The case $r < \mu_k^{-\frac{1}{2}}$ can be treated similarly and thus we omit the details. By the following asymptotic formulas \cite[p.133, p.242]{Asymptotics}
			\begin{align*}
				J_{\ell+\frac{n-2}{2}}(x)=\left ( \frac{2}{\pi x}\right )^{\frac 12}
				\Bigg[ & \sin\left ( x-\frac12\left ( \ell+\frac{n-3}{2} \right )\pi  \right ) \\
				& + \frac{A_1}{x}\cos \left ( x-\frac12\left ( \ell+\frac{n-3}{2} \right )\pi  \right )+O\left ( \frac{1}{x^2}\right ) \Bigg],
			\end{align*}
			\begin{align*}
				Y_{\ell+\frac{n-2}{2}}(x)=(-1)\left ( \frac{2}{\pi x}\right )^{\frac 12} \left [cos \left ( x-\frac12\left ( \ell+\frac{n-3}{2} \right )\pi  \right )+O\left ( \frac{1}{x}\right )\right ],
			\end{align*}
			where
			\begin{align*}
				A_1=\frac 18(2\ell+n-1)(2\ell+n-3),
			\end{align*}
			a direct calculation yields that if $r \ge \mu_k^{-\frac{1}{2}}$,
			\[
			\begin{aligned}
				\int_{0}^{r} J_{\ell+\frac{n-2}{2}}^{2}(t \sqrt{\mu_k}) t q(t) \, \mathrm{d} t = \frac{1}{\pi \mu_k^{\frac{1}{2}}} \int_{0}^{r} q(t) \, \mathrm{d} t + O\left( \frac{1}{\mu_k} \right).
			\end{aligned}
			\]
			and
			\[
			\int_{0}^{r} J_{\ell+\frac{n-2}{2}}(t \sqrt{\mu_k}) Y_{\ell+\frac{n-2}{2}}(t \sqrt{\mu_k}) t q(t) \, \mathrm{d} t = O\left( \frac{1}{\mu_k} \right).
			\]
			Substituting these results into \eqref{eq-bessel-2} we obtain
			\begin{align}\label{eq-phi}
				\phi(r, \mu_k) =& r^{\frac{1}{2}} J_{\ell+\frac{n-2}{2}}(r \sqrt{\mu_k}) - \frac{r^{\frac{1}{2}}}{2 \sqrt{\mu_k}} Y_{\ell+\frac{n-2}{2}}(r \sqrt{\mu_k}) \int_{0}^{r} q(t) \, \mathrm{d} t + O\left( \frac{1}{\mu_k^{\frac{5}{4}}} \right) \nonumber\\
				&= \frac{2^{\frac{1}{2}}}{\pi^{\frac{1}{2}} \mu_k^{\frac{1}{4}}} \left\{ \sin\left(r \sqrt{\mu_k} - \frac{1}{2} (\ell+\frac{n-3}{2})\pi\right) \right. \nonumber\\
				&\left. - \frac{h(r)}{\sqrt{\mu_k}} \cos\left(r \sqrt{\mu_k} - \frac{1}{2} (\ell+\frac{n-3}{2})\pi\right) + O\left( \frac{1}{\mu_k} \right) \right\},
			\end{align}
			where
			\[
			h(r) = \frac{1}{2} \int_{0}^{r} q(t) \, \mathrm{d} t - \frac{A_1}{r}.
			\]
			
			For $\phi'(r,\mu_k)$, we differentiate \eqref{eq-solution-bessel} with respect to $r$ and then, following the same process as \eqref{eq-solution-bessel}-\eqref{eq-phi}. This gives
			\begin{align}\label{eq-phi-diff}
				\phi'(r, \mu_k) &= \frac{2^{\frac{1}{2}} \mu_k^{\frac{1}{4}}}{\pi^{\frac{1}{2}}} \left\{ \cos\left(r \sqrt{\mu_k} - \frac{1}{2} (\ell+\frac{n-3}{2})\pi\right) \right. \nonumber\\
				&\quad\left. + \frac{h(r)}{\sqrt{\mu_k}} \sin\left(r \sqrt{\mu_k} - \frac{1}{2} (\ell+\frac{n-3}{2})\pi\right) + O\left( \frac{1}{\mu_k} \right) \right\}.
			\end{align}
			
			Substituting \eqref{eq-psi}, \eqref{eq-psi-diff}, \eqref{eq-phi} and \eqref{eq-phi-diff} into \eqref{eq-equation}, it follows that
			\begin{align}\label{eq-1823}
				&\cos\left(z + a\sqrt{\mu_k} - \frac{1}{2}(\ell+\frac{n-3}{2})\pi - \frac{1}{4}\pi\right) + \frac{h(a)}{\sqrt{\mu_k}} \sin\left(z + a\sqrt{\mu_k} - \frac{1}{2}(\ell+\frac{n-3}{2})\pi - \frac{1}{4}\pi\right) \nonumber\\
				&= O\left(\frac{1}{\mu_k}\right) + O\left(\frac{1}{z}\right).
			\end{align}
			Here $z + a\sqrt{\mu_k}$ is a steadily increasing function of $\mu_k$, hence, as this quantity ranges over any interval of the form
			\[
			\left( \bigl(\tfrac12 (\ell+\frac{n-3}{2}) + k + \tfrac14\bigr)\pi,\;
			\bigl(\tfrac12 (\ell+\frac{n-3}{2}) + k + \tfrac54\bigr)\pi \right),
			\quad k \text{ sufficiently large},
			\]
			the left-hand side of \eqref{eq-1823} changes sign. Therefore, each such interval contains at least one root and we define this root by setting
			\begin{align}\label{eq-1834}
				z + a\sqrt{\mu_k} = \left(\tfrac12 (\ell+\frac{n-3}{2}) + k + \tfrac34\right)\pi + \delta,
			\end{align}
			this, together with \eqref{eq-1823}, implies
			\[
			\sin\delta - \frac{h(a)}{\sqrt{\mu_k}} \cos\delta
			= O(\mu_k^{-1}) + O(z^{-1}),
			\]
			furthermore, we deduce that
			\begin{align}\label{eq-1843}
				\delta=\frac{h(a)}{\sqrt{\mu_k}}+O(\mu_k^{-1}) + O(z^{-1}).
			\end{align}
			It follows from \eqref{eq-1719}, \eqref{eq-1834} and \eqref{eq-1843} that
			\begin{align}\label{eq-1847}
				F_a(\mu_k) + a\sqrt{\mu_k}=\left(\tfrac12 (\ell+\frac{n-3}{2}) + k + \tfrac34\right)\pi+\frac{h(a)}{\sqrt{\mu_k}}+O(\mu_k^{-1}) + O(z^{-1}).
			\end{align}
			Following the same arguments in \cite{Titchmarsh1} 
			and \cite[pp. 160-163]{Titchmarsh2}, a direct calculation yields
			\begin{align*}
				z>AR\mu_k^{\frac{1}{2}},
			\end{align*}
			and
			\begin{align*}
				F_a(\mu_k)=\int_{0}^{R} \{\mu_k - q(r)\}^{\frac{1}{2}} \, dr-a\sqrt{\mu_k}+\frac{h(a)}{\sqrt{\mu_k}}+\delta_k,
			\end{align*}
			inserting these estimates into \eqref{eq-1847} yields \eqref{eq-eigenvalue-radial}. The proof of Proposition \ref{lemma-eigenvalue} is complete.
			\qed
			
		\end{appendix}

		\section*{Acknowledgements}
		S. Huang was supported by the National Natural Science Foundation of China (12171178)  and the Guangdong Basic and Applied Basic Research Foundation (2026B1515020075).

	\end{document}